\documentclass[reqno,12pt]{amsart} 

\newtheorem{theorem}{Theorem}
\newtheorem{proposition}{Proposition}
\newtheorem{corollary}{Corollary}

\newtheorem{lemma}{Lemma}

\theoremstyle{remark}
\newtheorem{remark}{Remark}
\newtheorem{definition}{Definition}

\numberwithin{equation}{section}

\allowdisplaybreaks

\author{Michael J.\ Schlosser}
\address{Fakult\"at f\"ur Mathematik, Universit\"at Wien,
Nordbergstra{\ss}e 15, A-1090 Vienna, Austria}
\email{michael.schlosser@univie.ac.at}
\urladdr{http://www.mat.univie.ac.at/{\textasciitilde}schlosse}
\thanks{Partly supported by FWF Austrian Science Fund
grant S9607 (which is part
of the Austrian National Research Network
``Analytic Combinatorics and Probabilistic Number Theory'').}

\title[A noncommutative weight-dependent binomial theorem]
{A noncommutative weight-dependent generalization of the binomial theorem}

\subjclass[2010]{Primary 16T30;
Secondary 05A30, 11B65, 33E05, 33E20}

\keywords{binomial theorem, commutation relations, symmetric functions,
$q$-commuting variables, elliptic-commuting variables,
elliptic binomial coefficient, elliptic hypergeometric series,
Frenkel and Turaev's ${}_{10}V_9$ summation}

\newcommand{\ta}{\theta}
\newcommand{\C}{\mathbb C}
\newcommand{\E}{\mathbb E}
\newcommand{\N}{\mathbb N}
\newcommand{\Z}{\mathbb Z}

\begin{document}

\begin{abstract}
A weight-dependent generalization of the binomial theorem
for noncommuting variables is presented. This result extends the
well-known binomial theorem for $q$-commuting variables by a generic
weight function depending on two integers.
For two special cases of the weight function,
in both cases restricting it to depend only on a single integer,
the noncommutative binomial theorem involves an expansion of
complete symmetric functions, and of elementary symmetric functions,
respectively. Another special case concerns
the weight function to be a suitably chosen elliptic
(i.e., doubly-periodic meromorphic) function, in which case
an elliptic generalization of the binomial theorem is obtained.
The latter is utilized to quickly recover Frenkel and Turaev's
elliptic hypergeometric ${}_{10}V_9$ summation formula, an identity
fundamental to the theory of elliptic hypergeometric series.
Further specializations yield noncommutative binomial theorems
of basic hypergeometric type. 
\end{abstract}

\maketitle

\section{Introduction}\label{secintro}

For an indeterminate $q$, let
$\C_q[x,y]$ be the associative unital algebra over $\C$
generated by $x$ and $y$, satisfying the relation
\begin{equation}\label{qcomm}
yx=qxy.
\end{equation}
$\C_q[x,y]$ can be regarded as a $q$-deformation of
the commutative algebra $\C[x,y]$.
The variables $x,y$ forming $\C_q[x,y]$ are referred to as
{\em $q$-commuting}
variables.

The following binomial theorem for $q$-commuting variables
is well known and usually attributed to M.P.~Sch\"utzenberger~\cite{Sb}.
However, for the case of $x$ and $y$ being $n\times n$ square matrices
with complex entries and $q$ (then necessarily) being a root of unity
(else \eqref{qcomm} cannot be satisfied), a proof, which extends
verbatim to the general case, was already given in 1950
by the Scottish mathematician H.S.A.~Potter~\cite{P}.
For an excellent account of the history of \eqref{qcomm}
in the context of matrix theory, see \cite{HMS}.

\begin{proposition}\label{qbinthm}
The following identity is valid in $\C_q[x,y]$:
\begin{equation}
(x+y)^n=\sum_{k=0}^n\begin{bmatrix}n\\k\end{bmatrix}_{q}x^ky^{n-k}.
\end{equation}
\end{proposition}
Here,
\begin{equation}
\begin{bmatrix}n\\k\end{bmatrix}_{q}:=\frac{(q;q)_n}{(q;q)_k(q;q)_{n-k}}
\end{equation}
is the {\em $q$-binomial coefficient}, defined for
nonnegative integers $n$ and $k$ with $n\ge k$, where, for
an indeterminant $a$, the {\em $q$-shifted factorial} is defined as
\begin{equation}
(a;q)_k:=(1-a)(1-aq)\dots(1-aq^{k-1})\qquad
\text{for}\quad k=0,1,2,\dots.
\end{equation}

Proposition~\ref{qbinthm} plays an important role in the theory of
quantum groups and $q$-special functions, see \cite{Kl,K}.

While the author's initial goal was to generalize
Proposition~\ref{qbinthm} to the ``elliptic case'', investigations
led to the discovery of a yet more general result,
a noncommutative binomial theorem involving
a generic weight function that depends on two integers. 

This paper is organized as follows.
In Section~\ref{secgenw} a noncommutative algebra is introduced
for a generic weight function depending on two integers.
The three commutation relations which define this algebra
are responsible for the validity of the noncommutative
binomial theorem. The chosen weight function uniquely determines
the corresponding binomial coefficients. These appear as coefficients
in the expansion of the noncommutative binomial theorem in
Theorem~\ref{wdbinth}, which is the main result of this paper.
This result can also be very nicely combinatorially interpreted
in terms of weighted lattice paths.
By multiplying two instances of the binomial theorem and suitably
taking coefficients, a convolution formula for the weight-dependent
binomial coefficients is deduced, while two other convolution formulae
are derived by means of the combinatorics of weighted lattice paths.
Section~\ref{secsym} focusses on two specific choices of the
weight function where (as is well-known) the binomial coefficients
become symmetric functions, namely complete symmetric
functions and elementary symmetric functions, respectively.
However, the noncommutative binomial theorems involving the
complete and elementary symmetric functions in
\eqref{hrel}/\eqref{symhbthm} and \eqref{erel}/\eqref{symebthm}
already appear to be new (which is quite surprising, given
its simplicity).
The situation is particularly interesting in
Section~\ref{secellbincoeffs} where an {\em elliptic}
(i.e., doubly-periodic meromorphic) weight function
is considered. In this case the noncommutative algebra can be
very elegantly described in terms of shifts on two of the variables.
The four variables forming this algebra are referred to as
{\em elliptic-commuting}
variables, while the coefficients appearing in the binomial
expansion of these variables are {\em elliptic binomial coefficients}
(which in fact are even totally elliptic functions).
The convolution formula for the latter turns out to be a variant
of Frenkel and Turaev's elliptic hypergeometric ${}_{10}V_9$
summation formula, an identity fundamental to the theory of
elliptic hypergeometric series.
While this suggests that the noncommutative elliptic binomial
theorem should be useful in the theory of elliptic hypergeometric
series, related elliptic special functions, and elliptic quantum groups
(see \cite[Sec.~11]{GR}, \cite{Sp2}, and \cite{ES}, respectively),
it is possible that the more general weight-dependent result in
Theorem~\ref{wdbinth}, or at least its symmetric function
specializations in \eqref{hrel}/\eqref{symhbthm} and
\eqref{erel}/\eqref{symebthm}, will be similarly
useful, e.g.\ in symmetric function theory in general or, speculatively,
even in the construction of quantum groups involving symmetric functions.
A further challenge of combinatorial-theoretic nature consists in finding
possible {\em weight-dependent} noncommutative extensions of MacMahon's
Master Theorem (which would maybe generalize the results of \cite{KP}).
Already working out an {\em elliptic} extension of MacMahon's
Master Theorem would be an exciting achievement.
In Appendix~\ref{appa} the whole set-up of Section~\ref{secgenw}
is extended by introducing an additional weight function, again
depending on two integers. This leads to a further extension of the
noncommutative algebra and corresponding binomial theorem by which
one in principle is able to consider more general cases
(which however is not pursued further in this paper).
Finally, in Appendix~\ref{appb}, particularly attractive
basic hypergeometric specializations of the elliptic case are
considered and made explicit. (This section may serve as a teaser.
Some readers, who are familiar with basic hypergeometric series,
may enjoy looking at Subsections~\ref{appbbal} and \ref{appbvwp}
first, and verify, say, the $n=2,3$ cases of the binomial theorems
in \eqref{bthbal} and \eqref{bthvwp} by hand.)

\smallskip
{\em Acknowledgements.} I would like to thank Tom Koornwinder for
private discussions on the problem of finding an elliptic extension of
the binomial theorem for $q$-commuting variables. These discussions
took place during the workshop on ``Elliptic integrable systems,
isomonodromy problems, and hypergeometric functions'' at the
Max Planck Institute for Mathematics in Bonn, July 21--25, 2008.
I would further like to thank Tom for his continued interest
and encouragement. I would also like to thank Johann Cigler for
fruitful discussions on the noncommutative binomial theorem.

The main results of this paper were presented at the
``Discrete Systems and Special Functions'' workshop
at the Isaac Newton Institute for Mathematical Sciences in Cambridge,
June~29 -- July~3, 2009. (The elliptic case was already presented
at several occasions before, the first time at a seminar at
Nagoya University on September 3, 2008.)
I am indebted to the organizers of both of these meetings
(Yu.I.~Manin, M.~Noumi, E.M.~Rains, H.~Rosengren, V.P.~Spiridonov, and
P.~Clarkson, R.~Halburd, M.~ Noumi, A.~O.~Daalhuis, respectively)
for inviting me to these workshops which have been highly
stimulating.

\section{Weight-dependent commutation relations
 and binomial theorem}\label{secgenw}

\subsection{A noncommutative algebra}\label{subsecncalg}
Let $\N$ and $\N_0$ denote the sets of positive
and nonnegative integers, respectively.
We will work in the following noncommutative algebra.
(A slight extension of this algebra is considered in Appendix~\ref{appa}.)

\begin{definition}\label{defwa}
For a doubly-indexed sequence of indeterminates $(w(s,t))_{s,t\in\N}$
let $\C_w[x,y]$ be the associative unital algebra over $\C$
generated by $x$ and $y$, satisfying the following three relations:
\begin{subequations}\label{defwaeq}
\begin{align}\label{1strel}
yx&=w(1,1)\,xy,\\\label{2ndrel}
x\,w(s,t)&=w(s+1,t)\,x,\\\label{3rdrel}
y\,w(s,t)&=w(s,t+1)\,y,
\end{align}
\end{subequations}
for all $(s,t)\in\N^2$.
\end{definition}

We refer to the $w(s,t)$, and more generally,
products (and even polynomials) of the $w(s,t)$,
as ``weights'' (in consideration of
the combinatorial interpretation in Subsection~\ref{secwdbc}).
Notice that for $w(s,t)=q$, for all $s,t\in\N$, for an indeterminate $q$,
$\C_w[x,y]$ reduces to $\C_q[x,y]$, the algebra of two
$q$-commuting variables considered in the Introduction.

The relations in \eqref{defwaeq} for generic weights $w(s,t)$
are indeed well-defined. Define the {\em canonical form} of an element
in $\C_w[x,y]$ to be of the form
\begin{equation*}
\sum_{k,l\ge 0}c_{k,l}x^ky^l,
\end{equation*}
where finitely many $c_{k,l}\neq0$.
Here the coefficients $c_{k,l}$ are elements of
$\C[(w(s,t))_{s,t\in\N}]$, the polynomial ring over $\C$
of the indeterminates $(w(s,t))_{s,t\in\N}$.
Since, according to \eqref{2ndrel} and \eqref{3rdrel}, $x$ and $y$
act as independent shift operators on the components of the weights,
it is straightforward to verify that the canonical form of an arbitrary
expression in $\C_w[x,y]$ is unique.
(In the terminology of \cite{B}, all elements of $\C_w[x,y]$ are
reduction-unique.)
This follows by induction (on the minimal number of commutation
relations needed to bring an expression into canonical form)
and observing that the simple expression $yx\,w(s,t)$
reduces to a unique canonical form regardless in which order
(e.g., $y$ and $x$, or $x$ and $w(s,t)$, are swapped first, etc.)
the commutation relations are applied for this purpose.

\subsection{Weight-dependent binomial coefficients}\label{secwdbc}

As before, we consider a doubly-in\-dexed sequence of indeterminate weights
$(w(s,t))_{s,t\in\N}$. For $s\in\N$ and $t\in\N_0$, we write
\begin{subequations}
\begin{equation}
W(s,t):=\prod_{j=1}^tw(s,j)
\end{equation}
(the empty product, which occurs when $t=0$, being defined as $1$)
for brevity. Note that for $s,t\in\N$ we have
\begin{equation}\label{w2W}
w(s,t)=\frac{W(s,t)}{W(s,t-1)}.
\end{equation}
\end{subequations}
To distinguish, we refer to the $W(s,t)$ as {\em big} weights,
and to the $w(s,t)$ as {\em small} weights.

Let the {\em weight-dependent binomial coefficients}
be defined by
\begin{subequations}\label{wbineq}
\begin{align}\label{recu}
&{}_{\stackrel{\phantom w}{\stackrel{\phantom w}w}}\!\!
\begin{bmatrix}0\\0\end{bmatrix}=1,\qquad
{}_{\stackrel{\phantom w}{\stackrel{\phantom w}w}}\!\!
\begin{bmatrix}n\\k\end{bmatrix}=0
\qquad\text{for\/ $n\in\N_0$, and\/
$k\in-\N$ or $k>n$},\\
\intertext{and}
\label{recw}
&{}_{\stackrel{\phantom w}{\stackrel{\phantom w}w}}\!\!
\begin{bmatrix}n+1\\k\end{bmatrix}=
{}_{\stackrel{\phantom w}{\stackrel{\phantom w}w}}\!\!
\begin{bmatrix}n\\k\end{bmatrix}
+{}_{\stackrel{\phantom w}{\stackrel{\phantom w}w}}\!\!
\begin{bmatrix}n\\k-1\end{bmatrix}
\,W(k,n+1-k)
\qquad \text{for $n,k\in\N_0$}.
\end{align}
\end{subequations}
The more general
{\em double} weight-dependent binomial coefficients involving
two generic weight functions are defined in Equation \eqref{vwbineq}.
To avoid possible misconception, it should be stressed that the
weight-dependent binomial coefficients in \eqref{wbineq} in general do
{\em not} satisfy the symmetry ${}_{\stackrel{\phantom w}w}\!
\left[\begin{smallmatrix}n\\k\end{smallmatrix}\right]=
{}_{\stackrel{\phantom w}w}\!
\left[\begin{smallmatrix}n\\n-k\end{smallmatrix}\right]$.
(In case they are symmetric, a second recursion from \eqref{recw}
can be immediately deduced, from which together with \eqref{recw}
a closed product formula for the weight-dependent binomial
coefficients can be derived.)
It follows from \eqref{wbineq} immediately by induction that
\begin{equation}\label{bordercond}
{}_{\stackrel{\phantom w}{\stackrel{\phantom w}w}}\!\!
\begin{bmatrix}n\\0\end{bmatrix}=
{}_{\stackrel{\phantom w}{\stackrel{\phantom w}w}}\!\!
\begin{bmatrix}n\\n\end{bmatrix}=1\qquad
\text{for\/ $n\in\N_0$}.
\end{equation}

The big weights $W(s,t)$ and the weight-dependent
binomial coefficients
${}_{\stackrel{\phantom w}w}\!
\left[\begin{smallmatrix}n\\k\end{smallmatrix}\right]$
have an elegant combinatorial
interpretation in terms of {\em weighted lattice paths}.
Consider lattice paths in the planar integer lattice consisting
of positively directed unit vertical and horizontal steps.
Such paths can be ``enumerated'' with respect to a generic weight
function $w$. In particular, assign the big weight $W(s,t)$
to each horizontal step $(s-1,t)\to(s,t)$

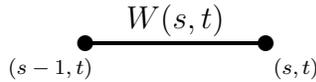
\centerline{
\unitlength2cm
\begin{picture}(1.1,.8)
\linethickness{1pt}
\put(0,0.4){\circle*{.1}}
\put(.05,0.3){\makebox(0,0)[tr]{{\tiny $(s-1,t)$}}}
\put(0,0.4){\line(1,0){1.2}}
\put(1.2,0.4){\circle*{.1}}
\put(1.25,0.3){\makebox(0,0)[tl]{{\tiny $(s,t)$}}}
\put(.6,0.45){\makebox(0,0)[b]{$W(s,t)$}}
\end{picture}}

\noindent and (for the moment) assign weight $1$
to each vertical step $(s,t-1)\to(s,t)$. (We will consider the more
general case of having an additional weight function $v$ defined
on the vertical steps in Appendix~\ref{appa}.) Further, define
the weight $w(P)$ of a path $P$ to be the product of the weights of
all its steps. For instance,  the following path $P_0(A\to\varOmega)$
from $A=(0,0)$ to $\varOmega=(5,2)$

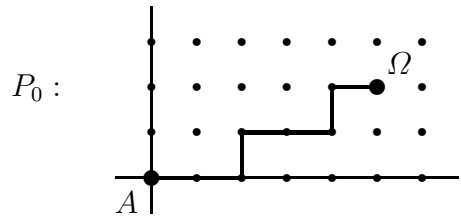
\centerline{
\unitlength.6cm
\begin{picture}(6,5.3)
\linethickness{1.2pt}
\put(0,1.1){\circle*{.35}}
\put(5,3.1){\circle*{.35}}
\put(-.3,.8){\makebox(0,0)[tr]{$A$}}
\put(5.2,3.4){\makebox(0,0)[bl]{$\varOmega$}}
\multiput(0,1.1)(1,0){7}{\circle*{.2}}
\multiput(0,2.1)(1,0){7}{\circle*{.2}}
\multiput(0,3.1)(1,0){7}{\circle*{.2}}
\multiput(0,4.1)(1,0){7}{\circle*{.2}}
\put(0,1.1){\line(1,0){2}}
\put(2,1.1){\line(0,1){1}}
\put(2,2.1){\line(1,0){2}}
\put(4,2.1){\line(0,1){1}}
\put(4,3.1){\line(1,0){1}}
\linethickness{.5pt}
\put(-.8,1.1){\line(1,0){7.6}}
\put(0,.3){\line(0,1){4.6}}
\put(-2,3.1){\makebox(0,0)[r]{$P_0:$}}
\end{picture}}

\noindent has weight
$$w(P_0)=1\cdot 1\cdot W(3,1)\cdot W(4,1)\cdot W(5,2)=
w(3,1)w(4,1)w(5,1)w(5,2).$$
(Equivalently, this corresponds to picking up the weights
$w(s,t)$ for each of the points $(s,t-1)$ strictly below the path,
for $s,t\ge 1$).
We will come back to this specific example shortly after the
proof of Theorem~\ref{wdbinth}.

Given two points $A,\varOmega\in\N_0^2$, let $\mathcal P(A\to\varOmega)$
be the set of all paths from $A$ to $\varOmega$. Further, define
\begin{equation*}
w(\mathcal P(A\to\varOmega)):=\sum_{P\in \mathcal P(A\to\varOmega)} w(P)
\end{equation*}
to be the {\em generating function} with respect to the
weight function $w$ of all paths from $A$ to $\varOmega$.
Now it is clear that
\begin{equation}\label{lpwbc}
w(\mathcal P((0,0)\to (k,n-k)))=
{}_{\stackrel{\phantom w}{\stackrel{\phantom w}w}}\!\!
\begin{bmatrix}n\\k\end{bmatrix}.
\end{equation}
Indeed, the path generating function satisfies the same
relations as the binomial coefficients in \eqref{wbineq}.
The initial conditions \eqref{recu} being clear,
the validity of the recursion \eqref{recw} stems from the fact
that the final step of a path ending in $(k,n+1-k)$ is either
{\em vertical} or {\em horizontal}.

In \cite{S}, lattice paths in $\Z^2$ were enumerated with respect to
the specific {\em elliptic} weight function
$w(s,t)=w_{a,b;q,p}(s,t)$ as defined in \eqref{wdef},
giving as generating functions
the (closed form) elliptic binomial coefficients
$\left[\begin{smallmatrix}n\\k\end{smallmatrix}\right]_{a,b;q,p}$
in \eqref{ellbc}. It was exactly this result which inspired the
search for the algebra $\C_w[x,y]$ and the binomial theorem
in Theorem~\ref{wdbinth}. We will take a closer look at
the elliptic case in Section~\ref{secellbincoeffs}.

\subsection{Noncommutative binomial theorem}

We have the following elegant result:
\begin{theorem}[Weight-dependent binomial theorem]\label{wdbinth}
Let $n\in\N_0$. Then the following identity is valid in $\C_w[x,y]$:
\begin{equation}\label{wdbinthid}
(x+y)^n=\sum_{k=0}^n\,
{}_{\stackrel{\phantom w}{\stackrel{\phantom w}w}}\!\!
\left[\begin{matrix}n\\k\end{matrix}\right]x^ky^{n-k}.
\end{equation}
\end{theorem}

\begin{proof}
We proceed by induction on $n$. For $n=0$ the formula is trivial.
Now let $n>0$ ($n$ being fixed) and assume that we have already
shown the formula
for all nonnegative integers $\le n$. We need to show
\begin{equation}
(x+y)^{n+1}=\sum_{k=0}^{n+1}\,
{}_{\stackrel{\phantom w}{\stackrel{\phantom w}w}}\!\!
\begin{bmatrix}n+1\\k\end{bmatrix}x^ky^{n+1-k}.
\end{equation}
By the recursion formula \eqref{recw},
the right-hand side is
\begin{align*}
\sum_{k=0}^{n+1}\,
&{}_{\stackrel{\phantom w}{\stackrel{\phantom w}w}}\!\!
\begin{bmatrix}n\\k\end{bmatrix}x^ky^{n+1-k}
+\sum_{k=0}^{n+1}\,
{}_{\stackrel{\phantom w}{\stackrel{\phantom w}w}}\!\!
\begin{bmatrix}n\\k-1\end{bmatrix}
W(k,n+1-k)\,x^ky^{n+1-k}\\
=&\sum_{k=0}^n\,
{}_{\stackrel{\phantom w}{\stackrel{\phantom w}w}}\!\!
\begin{bmatrix}n\\k\end{bmatrix}x^ky^{n-k}y
+\sum_{k=0}^n\,
{}_{\stackrel{\phantom w}{\stackrel{\phantom w}w}}\!\!
\begin{bmatrix}n\\k\end{bmatrix}W(k+1,n-k)\,x^{k+1}y^{n-k},
\end{align*}
where the range of summation in each of the sums was delimited
due to \eqref{recu}. It remains to be shown that
\begin{equation}\label{star}
W(k+1,n-k)\,x^{k+1}y^{n-k}=x^ky^{n-k}\,x.
\end{equation}
In particular, using \eqref{1strel} (and induction) we have
\begin{equation*}
W(k+1,n-k)\,x^{k+1}y^{n-k}=
x^kW(1,n-k)\,xy^{n-k},
\end{equation*}
so \eqref{star} is shown as soon as one establishes
\begin{equation}\label{tbs1}
W(1,n-k)\,xy^{n-k}=y^{n-k}x.
\end{equation}
The left-hand side is
\begin{multline}
\bigg(\prod_{j=1}^{n-k}w(1,j)\bigg)\,xyy^{n-k-1}
=\bigg(\prod_{j=2}^{n-k}w(1,j)\bigg)\,yxy^{n-k-1}\
=y\bigg(\prod_{j=1}^{n-k-1}w(1,j)\bigg)\,xy^{n-k-1}\\\label{tbs2}
=y\,W(1,n-k-1)\,xy^{n-k-1},
\end{multline}
where we have first used 
\eqref{1strel} and then \eqref{3rdrel}.
The identity \eqref{tbs1} follows now from \eqref{tbs2}
immediately by induction on $n-k\ge 0$.
\end{proof}

Coming back to the interpretation of the weight-dependent
binomial coefficients as generating functions for weighted lattice
paths, see Subsection~\ref{secwdbc}, the expansion \eqref{wdbinthid} in
Theorem~\ref{wdbinth} itself has an accordingly nice interpretation. 
The identification of expressions (or ``words'' consisting of
concatenated symbols) in $\C_w[x,y]$ and lattice paths in $\Z^2$
works locally (variable by variable, or step by step) as follows:
\begin{align*}
x\quad&\longleftrightarrow\quad\text{horizontal step},\\
y \quad&\longleftrightarrow\quad\text{vertical step}. 
\end{align*}
Under this correspondence $xy$
means that a horizontal step is followed by a vertical step,
while ``$yx$'' means that a vertical step is followed by a horizontal step
(we read from left to right). The relations of the algebra
$\C_w[x,y]$ in Definition~\ref{defwa} exactly take into account
the changes of the respective weights when consecutive horizontal
and vertical steps are being interchanged. For instance,
the specific path $P_0$ considered in Subsection~\ref{secwdbc}
corresponds to the algebraic expression
\begin{multline*}
x\,x\,y\,x\,x\,y\,x=x\,x\,w(1,1)\,x\,y\,x\,w(1,1)\,x\,y
=w(3,1)w(5,2)\,x^3\,w(1,1)\,x\,y\,x\,y\\
=w(3,1)w(5,2)w(4,1)\,x^4\,w(1,1)\,x\,y^2
=w(3,1)w(5,2)w(4,1)w(5,1)\,x^5y^2=w(P_0)\,x^5y^2,
\end{multline*}
where the left coefficient (of the double monomial $x^5y^2$) in the
canonical form of the algebraic expression is the weight of the path.
Concluding, the left-hand side of \eqref{wdbinthid}, i.e.\
$(x+y)^n$, translates into paths of length $n$ 
having for each step a choice of going in horizontal or vertical 
positive direction,
while the right-hand side of \eqref{wdbinthid} refines the counting
according to the number of horizontal steps $k$ (for $0\le k\le n$).

\subsection{Convolution formulae}

We are ready to apply the weight-dependent binomial theorem
in Theorem~\ref{wdbinth}
to derive weight-dependent extensions of the well-known (Vandermonde)
convolution of binomial coefficients,
\begin{equation}
\binom{n+m}k=\sum_j\binom nj\binom m{k-j}.
\end{equation}

The following corollaries, although being derived with the help of
noncommuting variables, themselves concern identities of
{\em commuting} elements.

\begin{corollary}[First weight-dependent binomial convolution
formula]\label{wdconv1}
Let $n$, $m$, and $k$ be nonnegative integers.
For the binomial coefficients in \eqref{wbineq}, defined by
the doubly-indexed sequence of indeterminate weights
$(w(s,t))_{s,t\in\N}$, we have the following
formal identity in $\C[(w(s,t))_{s,t\in\N}]$:
\begin{equation}\label{wdconv1id}
{}_{\stackrel{\phantom w}{\stackrel{\phantom w}w}}\!\!
\begin{bmatrix}n+m\\k\end{bmatrix}=\sum_{j=0}^{\min(k,n)}
{}_{\stackrel{\phantom w}{\stackrel{\phantom w}w}}\!\!
\begin{bmatrix}n\\j\end{bmatrix}
\left(x^jy^{n-j}\,
{}_{\stackrel{\phantom w}{\stackrel{\phantom w}w}}\!\!
\begin{bmatrix}m\\k-j\end{bmatrix}
y^{j-n}x^{-j}\right)
\prod_{i=1}^{k-j}W(i+j,n-j).
\end{equation}
\end{corollary}

The above identity in $\C[(w(s,t))_{s,t\in\N}]$ is {\em formal}
as it contains the expression
\begin{equation}
x^jy^{n-j}
{}_{\stackrel{\phantom w}{\stackrel{\phantom w}w}}\!\!
\begin{bmatrix}m\\k-j\end{bmatrix}
y^{j-n}x^{-j}
\end{equation}
(which evaluates to an expression in $\C[(w(s,t))_{s,t\in\N}]$)
in the summand, where $x,y\notin \C[(w(s,t))_{s,t\in\N}]$.
The $x$ and $y$ are understood to be shift operators
as defined in \eqref{2ndrel} and \eqref{3rdrel}.
Formally, $x^jy^{n-j}$ has to commute with
${}_{\stackrel{\phantom w}w}\!
\left[\begin{smallmatrix}m\\k-j\end{smallmatrix}\right]$
which will involve various shifts of the weight functions
implicitly appearing in the $w$-binomial coefficient.
Afterwards $x^jy^{n-j}$ will cancel with its formal inverse
$y^{j-n}x^{-j}$.

First we state a useful lemma.
\begin{lemma}\label{lem}
The following identity holds in $\C_w[x,y]$.
\begin{equation}\label{lemeq}
y^kx^l=
\bigg(\prod_{i=1}^lW(i,k)\bigg)\,x^ly^k,\qquad\quad
\text{for\/ $k,l\in\N_0$.}
\end{equation}
\end{lemma}
\begin{proof}
The $k=0$ or $l=0$ cases are trivial. For $k,l\ge 1$
the identity \eqref{lemeq} follows straightforward by double induction
starting with the $k=l=1$ case which is \eqref{1strel}.
To prove the validity of \eqref{lemeq} for $k=1$ and $l>1$,
assume that $yx^s=\big(\prod_{i=1}^sW(i,1)\big)x^sy$
has already been shown for $1\le s <l$. Then
\begin{align*}
yx^l&=w(1,1)\,xyx^{l-1}=
W(1,1)\,x
\bigg(\prod_{i=1}^{l-1}W(i,1)\bigg)x^{l-1}y\\
&=\bigg(W(1,1)\prod_{i=1}^{l-1}W(i+1,1)\bigg)x^ly
=\bigg(\prod_{i=1}^lW(i,1)\bigg)\,x^ly,
\end{align*}
as to be shown. Now, for fixed $k,l\ge 1$, assume that
$y^tx^l=\big(\prod_{i=1}^lW(i,t)\big)x^ly^t$ has already been shown
for $1\le t <k$. Then
\begin{align*}
y^kx^l&=y\bigg(\prod_{i=1}^lW(i,k-1)\bigg)x^ly^{k-1}=
y\bigg(\prod_{i=1}^l\prod_{j=1}^{k-1}w(i,j)\bigg)x^ly^{k-1}\\&=
\bigg(\prod_{i=1}^l\prod_{j=1}^{k-1}w(i,j+1)\bigg)yx^ly^{k-1}
=\bigg(\prod_{i=1}^l\prod_{j=1}^kw(i,j)\bigg)x^ly^k
=\bigg(\prod_{i=1}^lW(i,k)\bigg)\,x^ly^k,
\end{align*}
 which establishes the lemma.
\end{proof}

\begin{proof}[Proof of Corollary~\ref{wdconv1}]
Working in $\C_w[x,y]$, one expands $(x+y)^{n+m}$
in two different ways and suitably extracts coefficients.
On the one hand,
\begin{equation}\label{eqwbinth1}
(x+y)^{n+m}=\sum_{k=0}^{n+m}\,
{}_{\stackrel{\phantom w}{\stackrel{\phantom w}w}}\!\!
\begin{bmatrix}n+m\\k\end{bmatrix}
x^ky^{n+m-k}.
\end{equation}
On the other hand,
\begin{align}
(x+y)^{n+m}{}&=(x+y)^n\,(x+y)^m\notag\\
&=\sum_{j=0}^n\sum_{l=0}^m\,
{}_{\stackrel{\phantom w}{\stackrel{\phantom w}w}}\!\!
\begin{bmatrix}n\\j\end{bmatrix}x^jy^{n-j}
{}_{\stackrel{\phantom w}{\stackrel{\phantom w}w}}\!\!
\begin{bmatrix}m\\l\end{bmatrix}x^ly^{m-l}\notag\\\label{eqwbinth2}
&=\sum_{j=0}^n\sum_{l=0}^m\,
{}_{\stackrel{\phantom w}{\stackrel{\phantom w}w}}\!\!
\begin{bmatrix}n\\j\end{bmatrix}x^jy^{n-j}
{}_{\stackrel{\phantom w}{\stackrel{\phantom w}w}}\!\!
\begin{bmatrix}m\\l\end{bmatrix}
y^{j-n}x^{-j}\,x^jy^{n-j}x^ly^{m-l}.
\end{align}
Now use Lemma~\ref{lem} to apply
\begin{equation*}
x^jy^{n-j}x^ly^{m-l}
=\bigg(\prod_{i=1}^lW(i+j,n-j)\bigg)x^{j+l}y^{n+m-j-l}\quad
\text{for\/ $n\ge j$},
\end{equation*}
and extract and equate (left) coefficients of
$x^ky^{n+m-k}$ in \eqref{eqwbinth1} and \eqref{eqwbinth2}.
We thus immediately establish the convolution formula
\eqref{wdconv1id}. 
\end{proof}

In terms of interpreting the weight-dependent binomial coefficients
as generating functions for weighted lattice paths
(see Eq.~\eqref{lpwbc}), the identity \eqref{wdconv1id}
translates to a convolution of paths with respect to a {\em diagonal},

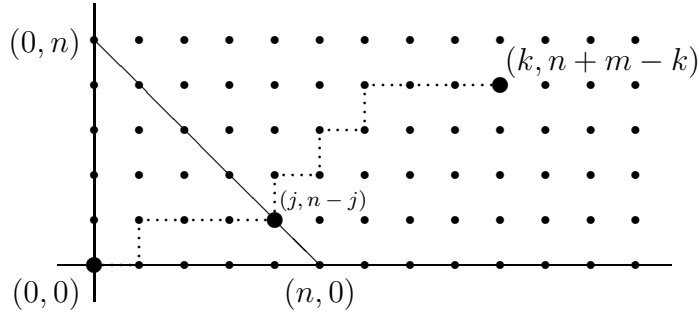
\centerline{
\unitlength.6cm
\begin{picture}(12,7.4)
\linethickness{1.2pt}
\put(0,1.1){\circle*{.35}}
\put(4,2.1){\circle*{.35}}
\put(9,5.1){\circle*{.35}}
\put(-.3,.8){\makebox(0,0)[tr]{$(0,0)$}}
\put(-.3,6){\makebox(0,0)[r]{$(0,n)$}}
\put(5,.8){\makebox(0,0)[t]{$(n,0)$}}
\put(9.1,5.25){\makebox(0,0)[bl]{$(k,n+m-k)$}}
\put(4.1,2.3){\makebox(0,0)[bl]{\tiny $(j,n-j)$}}
\multiput(0,1.1)(1,0){13}{\circle*{.2}}
\multiput(0,2.1)(1,0){13}{\circle*{.2}}
\multiput(0,3.1)(1,0){13}{\circle*{.2}}
\multiput(0,4.1)(1,0){13}{\circle*{.2}}
\multiput(0,5.1)(1,0){13}{\circle*{.2}}
\multiput(0,6.1)(1,0){13}{\circle*{.2}}
\multiput(0,1.1)(.2,0){5}{\circle*{.08}}
\multiput(1,1.1)(0,.2){5}{\circle*{.08}}
\multiput(1,2.1)(.2,0){15}{\circle*{.08}}
\multiput(4,2.1)(0,.2){5}{\circle*{.08}}
\multiput(4,3.1)(.2,0){5}{\circle*{.08}}
\multiput(5,3.1)(0,.2){5}{\circle*{.08}}
\multiput(5,4.1)(.2,0){5}{\circle*{.08}}
\multiput(6,4.1)(0,.2){5}{\circle*{.08}}
\multiput(6,5.1)(.2,0){15}{\circle*{.08}}
\linethickness{.5pt}
\put(0,6.1){\line(1,-1){5}}
\put(-.8,1.1){\line(1,0){13.6}}
\put(0,.3){\line(0,1){6.6}}
\end{picture}}

\noindent where the corresponding generating function identity is
immediately seen to be
\begin{align}\notag
&w\big(\mathcal P((0,0)\to(k,n+m-k))\big)\\\label{convd}
&=\sum_{j=0}^{\min(k,n)}w\big(\mathcal P((0,0)\to(j,n-j))\big)\,
w\big(\mathcal P((j,n-j)\to(k,n+m-k))\big).
\end{align}

One can also consider convolution with respect
to a {\em vertical line},

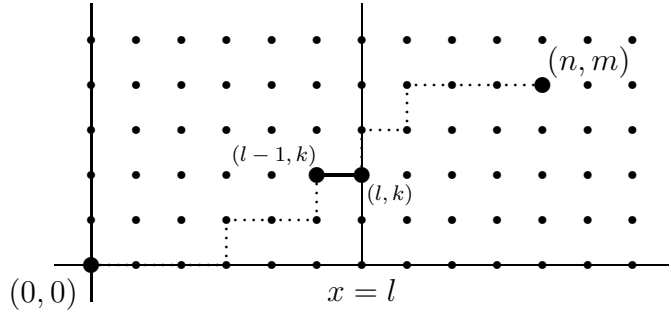
\centerline{
\unitlength.6cm
\begin{picture}(12,7.4)
\linethickness{1.2pt}
\put(0,1.1){\circle*{.35}}
\put(5,3.1){\circle*{.35}}
\put(6,3.1){\circle*{.35}}
\put(10,5.1){\circle*{.35}}
\put(-.3,.8){\makebox(0,0)[tr]{$(0,0)$}}
\put(5,3.3){\makebox(0,0)[br]{\tiny $(l-1,k)$}}
\put(6.1,2.9){\makebox(0,0)[tl]{\tiny $(l,k)$}}
\put(10.1,5.25){\makebox(0,0)[bl]{$(n,m)$}}
\put(6,.8){\makebox(0,0)[t]{$x=l$}}
\multiput(0,1.1)(1,0){13}{\circle*{.2}}
\multiput(0,2.1)(1,0){13}{\circle*{.2}}
\multiput(0,3.1)(1,0){13}{\circle*{.2}}
\multiput(0,4.1)(1,0){13}{\circle*{.2}}
\multiput(0,5.1)(1,0){13}{\circle*{.2}}
\multiput(0,6.1)(1,0){13}{\circle*{.2}}
\multiput(0,1.1)(.2,0){15}{\circle*{.08}}
\multiput(3,1.1)(0,.2){5}{\circle*{.08}}
\multiput(3,2.1)(.2,0){10}{\circle*{.08}}
\multiput(5,2.1)(0,.2){5}{\circle*{.08}}
\multiput(6,3.1)(0,.2){5}{\circle*{.08}}
\multiput(6,4.1)(.2,0){5}{\circle*{.08}}
\multiput(7,4.1)(0,.2){5}{\circle*{.08}}
\multiput(7,5.1)(.2,0){15}{\circle*{.08}}
\put(5,3.1){\line(1,0){1}}
\linethickness{.5pt}
\put(-.8,1.1){\line(1,0){13.6}}
\put(0,.3){\line(0,1){6.6}}
\put(6,1.1){\line(0,1){5.8}}
\end{picture}}

\noindent which corresponds to the identity
\begin{align}\notag
&w\big(\mathcal P((0,0)\to(n,m))\big)\\\label{convv}
&=\sum_{k=0}^mw\big(\mathcal P((0,0)\to(l-1,k))\big)\,
W(l,k)\,w\big(\mathcal P((l,k)\to(n,m))\big).
\end{align}
Here, $l$ is fixed ($1\le l\le n$) while the nonnegative integer $k$
is uniquely determined by the height of the path when it reaches
the vertical line $x=l$ first.

In terms of our weights $w$, we have the following result.

\begin{corollary}[Second weight-dependent binomial convolution
formula]\label{wdconv2}
Let $n$, $m$, and $k$ be nonnegative integers with $1\le l\le n$.
For the binomial coefficients in \eqref{wbineq}, defined by
the doubly-indexed sequence of indeterminate weights
$(w(s,t))_{s,t\in\N}$, we have the following
formal identity in $\C[(w(s,t))_{s,t\in\N}]$:
\begin{equation}\label{wdconv2id}
{}_{\stackrel{\phantom w}{\stackrel{\phantom w}w}}\!\!
\begin{bmatrix}n+m\\n\end{bmatrix}=\sum_{k=1}^{m}
{}_{\stackrel{\phantom w}{\stackrel{\phantom w}w}}\!\!
\begin{bmatrix}k+l-1\\l-1\end{bmatrix}
\left(x^ly^k\,
{}_{\stackrel{\phantom w}{\stackrel{\phantom w}w}}\!\!
\begin{bmatrix}n+m-l-k\\n-l\end{bmatrix}
y^{-k}x^{-l}\right)
\prod_{i=0}^{n-l}W(i+l,k).
\end{equation}
\end{corollary}

\begin{proof}
Translating the generating function identity \eqref{convv}
into an identity in $\C_w[x,y]$ with weight function $w$,
one immediately obtains
\begin{equation*}
{}_{\stackrel{\phantom w}{\stackrel{\phantom w}w}}\!\!
\begin{bmatrix}n+m\\n\end{bmatrix}x^ny^m=
\sum_{k=0}^m\,{}_{\stackrel{\phantom w}{\stackrel{\phantom w}w}}\!\!
\begin{bmatrix}k+l-1\\l-1\end{bmatrix}\,W(l,k)\,x^{l-1}y^k\,
x\,{}_{\stackrel{\phantom w}{\stackrel{\phantom w}w}}\!\!
\begin{bmatrix}n+m-l-k\\n-l\end{bmatrix}\,x^{n-l}y^{m-k}.
\end{equation*}
The further analysis is now similar to the proof of
Corollary~\ref{wdconv1}. The $y$'s have to be moved to the
most right, then the $x$'s, hereby creating weights and shifts,
while in the end the left coefficients of $x^ny^m$ have to be extracted
and equated to establish \eqref{wdconv2id}.  
\end{proof}

Finally, one can also consider convolution with respect
to a {\em horizontal line},

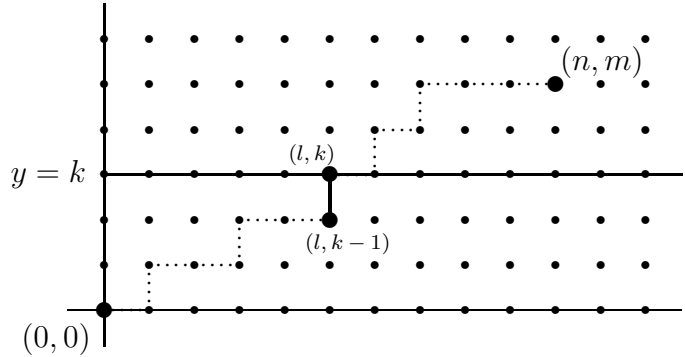
\centerline{
\unitlength.6cm
\begin{picture}(12,8.4)
\linethickness{1.2pt}
\put(0,1.1){\circle*{.35}}
\put(5,3.1){\circle*{.35}}
\put(5,4.1){\circle*{.35}}
\put(10,6.1){\circle*{.35}}
\put(-.3,.8){\makebox(0,0)[tr]{$(0,0)$}}
\put(5.4,2.8){\makebox(0,0)[t]{\tiny $(l,k-1)$}}
\put(4.6,4.3){\makebox(0,0)[b]{\tiny $(l,k)$}}
\put(10.1,6.25){\makebox(0,0)[bl]{$(n,m)$}}
\put(-.4,4.1){\makebox(0,0)[r]{$y=k$}}
\multiput(0,1.1)(1,0){13}{\circle*{.2}}
\multiput(0,2.1)(1,0){13}{\circle*{.2}}
\multiput(0,3.1)(1,0){13}{\circle*{.2}}
\multiput(0,4.1)(1,0){13}{\circle*{.2}}
\multiput(0,5.1)(1,0){13}{\circle*{.2}}
\multiput(0,6.1)(1,0){13}{\circle*{.2}}
\multiput(0,7.1)(1,0){13}{\circle*{.2}}
\multiput(0,1.1)(.2,0){5}{\circle*{.08}}
\multiput(1,1.1)(0,.2){5}{\circle*{.08}}
\multiput(1,2.1)(.2,0){10}{\circle*{.08}}
\multiput(3,2.1)(0,.2){5}{\circle*{.08}}
\multiput(3,3.1)(.2,0){10}{\circle*{.08}}
\multiput(5,4.1)(.2,0){5}{\circle*{.08}}
\multiput(6,4.1)(0,.2){5}{\circle*{.08}}
\multiput(6,5.1)(.2,0){5}{\circle*{.08}}
\multiput(7,5.1)(0,.2){5}{\circle*{.08}}
\multiput(7,6.1)(.2,0){15}{\circle*{.08}}
\put(5,3.1){\line(0,1){1}}
\linethickness{.5pt}
\put(-.8,1.1){\line(1,0){13.6}}
\put(0,.3){\line(0,1){7.6}}
\put(0,4.1){\line(1,0){12.8}}
\end{picture}}

\noindent which corresponds to the identity
\begin{align}\notag
&w\big(\mathcal P((0,0)\to(n,m))\big)\\\label{convh}
&=\sum_{l=0}^nw\big(\mathcal P((0,0)\to(l,k-1))\big)\,
w\big(\mathcal P((l,k)\to(n,m))\big).
\end{align}
Here, $k$ is fixed ($1\le k\le m$) while the nonnegative integer $l$
is the uniquely determined abscissa of the path when it reaches
the horizontal line $y=k$ first. 

In terms of our weights $w$, we have the following result.

\begin{corollary}[Third weight-dependent binomial convolution
formula]\label{wdconv3}
Let $n$, $m$, and $k$ be nonnegative integers with $1\le k\le m$.
For the binomial coefficients in \eqref{wbineq}, defined by
the doubly-indexed sequence of indeterminate weights
$(w(s,t))_{s,t\in\N}$, we have the following
formal identity in $\C[(w(s,t))_{s,t\in\N}]$:
\begin{equation}\label{wdconv3id}
{}_{\stackrel{\phantom w}{\stackrel{\phantom w}w}}\!\!
\begin{bmatrix}n+m\\n\end{bmatrix}=\sum_{l=0}^{n}
{}_{\stackrel{\phantom w}{\stackrel{\phantom w}w}}\!\!
\begin{bmatrix}l+k-1\\l\end{bmatrix}
\left(x^ly^k\,
{}_{\stackrel{\phantom w}{\stackrel{\phantom w}w}}\!\!
\begin{bmatrix}n+m-l-k\\n-l\end{bmatrix}
y^{-k}x^{-l}\right)
\prod_{i=1}^{n-l}W(i+l,k).
\end{equation}
\end{corollary}
\begin{proof}
Translating the generating function identity \eqref{convh}
into an identity in $\C_w[x,y]$ with weight function $w$,
one immediately obtains
\begin{equation*}
{}_{\stackrel{\phantom w}{\stackrel{\phantom w}w}}\!\!
\begin{bmatrix}n+m\\n\end{bmatrix}x^ny^m=
\sum_{l=0}^n\,{}_{\stackrel{\phantom w}{\stackrel{\phantom w}w}}\!\!
\begin{bmatrix}l+k-1\\l\end{bmatrix}\,x^{l}y^{k-1}y\,
{}_{\stackrel{\phantom w}{\stackrel{\phantom w}w}}\!\!
\begin{bmatrix}n+m-l-k\\n-l\end{bmatrix}\,x^{n-l}y^{m-k}.
\end{equation*}
The convolution formula \eqref{wdconv3id} follows immediately
after application of Lemma~\ref{lem} (where we have $k\ge 1$,
thus do not have to distinguish cases),
moving first the $y$'s, then the $x$'s,
to the most right and finally extracting and equating
the left coefficients of $x^ny^m$.  
\end{proof}

\section{Symmetric functions}\label{secsym}

Here we take a closer look at two important specializations of the
weights $w(s,t)$, both involving symmetric functions.
(See \cite{M} for a classic text book on symmetric function theory).

\subsection{Complete symmetric functions}
The first choice is $w(s,t)=a_t/a_{t-1}$.
In this case we have (essentially corresponding to
the $h$-labeling of lattice paths in \cite[Sec.~4.5]{Sa})
$$
{}_{\stackrel{\phantom w}{\stackrel{\phantom w}w}}\!\!
\left[\begin{matrix}n\\k\end{matrix}\right]=
h_k(a_0,a_1,\dots,a_{n-k})\,a_0^{-k} ,
$$
where $h_k(a_0,a_1,\dots,a_m)$ is the {\em complete symmetric function}
of order $k$,
defined by
\begin{align*}
h_0(a_0,a_1,\dots,a_m)&:=1
\intertext{and}
h_k(a_0,a_1,\dots,a_m)&:=
\sum_{0\le j_1\le j_2\le\dots\le j_k\le m}a_{j_1}a_{j_2}\dots a_{j_k}
\qquad\text{for}\quad k>0.
\end{align*}
Indeed, the complete symmetric functions satisfy the recursion
\begin{equation*}
h_k(a_0,a_1,\dots,a_{m+1})=h_k(a_0,a_1,\dots,a_m)+
a_{m+1}h_{k-1}(a_0,a_1,\dots,a_{m+1}),
\end{equation*}
which readily follows from specializing the recursion \eqref{wbineq}
for the weight-dependent binomial coefficients.

The relations of the algebra $\C_w[x,y]$ in Definition~\ref{defwa}
now reduce to
\begin{subequations}\label{hrel}
\begin{align}
yx&=\frac{a_1}{a_0}xy,\\\label{2ndhrel}
x\frac{a_t}{a_{t-1}}&=\frac{a_t}{a_{t-1}}x,\\\label{3rdhrel}
y\frac{a_t}{a_{t-1}}&=\frac{a_{t+1}}{a_t}y,
\end{align}
\end{subequations}
for all $t\in\N$.

The noncommutative binomial theorem in Theorem~\ref{wdbinth} now becomes
\begin{equation}\label{symhbthm}
(x+y)^n=\sum_{k=0}^n h_k(a_0,a_1,\dots,a_{n-k})\,a_0^{-k}\,x^k y^{n-k},
\end{equation}
which, despite its simplicity, appears to be new.

{}From Corollary~\ref{wdconv1} and \eqref{2ndhrel}/\eqref{3rdhrel}
the convolution
\begin{subequations}
\begin{equation}\label{h1}
h_k(a_0,a_1,\dots,a_{n+m-k})=
\sum_{j=0}^{\min(k,n)}h_j(a_0,\dots,a_{n-j})\,h_{k-j}(a_{n-j},\dots,a_{n+m-k})
\end{equation}
is obtained. On the other hand, Corollaries~\ref{wdconv2} and
\ref{wdconv3} reduce to the identities
\begin{equation}\label{h2}
h_n(a_0,a_1,\dots,a_m)=
\sum_{k=0}^mh_{l-1}(a_0,\dots,a_k)\,a_k\,h_{n-l}(a_k,\dots,a_m),\;
\text{for a fixed\/ $1\le l\le n$},
\end{equation}
and
\begin{equation}\label{schurh}
h_n(a_0,a_1,\dots,a_m)=
\sum_{l=0}^nh_{l}(a_0,\dots,a_{k-1})\,h_{n-l}(a_k,\dots,a_m),\;
\text{for a fixed\/ $1\le k\le m$},
\end{equation}
\end{subequations}
respectively. The identity in \eqref{schurh} is the special case
of a well-known convolution formula for Schur functions
\cite[p.~72, Eq.~(5.10)]{M}, for which the indexing partitions
are reduced to at most one row. The other two identities, \eqref{h1}
and \eqref{h2}, are most likely already known as well, although
the author has not been able to find them explicitly in the literature.

\subsection{Elementary symmetric functions}
Another choice is $w(s,t)=a_{s+t}/a_{s+t-1}$.
In this case we have (essentially corresponding
to the $e$-labeling of lattice paths in \cite[Sec.~4.5]{Sa})
$$
{}_{\stackrel{\phantom w}{\stackrel{\phantom w}w}}\!\!
\left[\begin{matrix}n\\k\end{matrix}\right]=
\frac{e_k(a_1,\dots,a_n)}{a_1\dots a_k},
$$
where $e_k(a_1,\dots,a_n)$ is the {\em elementary symmetric function}
of order $k$,
defined by
\begin{align*}
e_0(a_1,\dots,a_n)&:=1\\
\intertext{and}
e_k(a_1,\dots,a_n)&:=
\sum_{1\le j_1< j_2<\dots<j_k\le n}a_{j_1}a_{j_2}\dots a_{j_k}
\qquad\text{for}\quad k>0.
\end{align*}
Indeed, the elementary symmetric functions satisfy the recursion
\begin{equation*}
e_k(a_1,\dots,a_{n+1})=e_k(a_1,\dots,a_n)+
a_{n+1}e_{k-1}(a_1,\dots,a_n),
\end{equation*}
which again readily follows from specializing \eqref{wbineq}.

The relations of the algebra $\C_w[x,y]$ in Definition~\ref{defwa}
now reduce to
\begin{subequations}\label{erel}
\begin{align}
yx&=\frac{a_2}{a_1}xy,\\
x\frac{a_{s+t}}{a_{s+t-1}}&=\frac{a_{s+t+1}}{a_{s+t}}x,\\\label{3rderel}
y\frac{a_{s+t}}{a_{s+t-1}}&=\frac{a_{s+t+1}}{a_{s+t}}y,
\end{align}
\end{subequations}
for all $s,t\in\N$.

The noncommutative binomial theorem in Theorem~\ref{wdbinth} now becomes
\begin{equation}\label{symebthm}
(x+y)^n=\sum_{k=0}^n\frac{e_k(a_1,\dots,a_n)}{a_1\dots a_k}\,x^k y^{n-k},
\end{equation}
which, despite its simplicity, appears to be new.

The convolutions in Corollaries~\ref{wdconv1}, \ref{wdconv2}, and
\ref{wdconv3}, respectively, give
\begin{subequations}
\begin{equation}\label{schure}
e_k(a_1,a_2,\dots,a_{n+m})=
\sum_{j=0}^{\min(k,n)}e_j(a_1,\dots,a_n)\,e_{k-j}(a_{n+1},\dots,a_{n+m}),
\end{equation}
\begin{multline}\label{e1}
e_n(a_1,a_2,\dots,a_{n+m})=
\sum_{k=0}^me_{l-1}(a_1,\dots,a_{l+k-1})\,a_{l+k}\,
e_{n-l}(a_{l+k+1},\dots,a_{n+m}),\\
\text{for a fixed\/ $1\le l\le n$},
\end{multline}
and
\begin{multline}\label{e2}
e_n(a_1,a_2,\dots,a_{n+m})=
\sum_{l=0}^ne_l(a_1,\dots,a_{l+k-1})\,e_{n-l}(a_{l+k+1},\dots,a_{n+m}),\\
\text{for a fixed\/ $1\le k\le m$}.
\end{multline}
\end{subequations}
The identity in \eqref{schure} is another special case
(compare with \eqref{schurh})
of the convolution formula for Schur functions in
\cite[p.~72, Eq.~(5.10)]{M}, for which the indexing partitions
are now reduced to at most one {\em column}.
The other two identities, \eqref{e1} and \eqref{e2}, are most
likely already known as well, although the author
has not been able to find them explicitly in the literature.

\section{Elliptic hypergeometric series}\label{secellbincoeffs}

In this section we concentrate on the so-called {\em elliptic} case.
It was this case which, for the author, served as a motivation
to look out for generalizions of the $q$-commuting variables \eqref{qcomm}.

We explain some important notions from the theory of elliptic
hypergeometric series we will make use of (see also \cite[Ch.~11]{GR}).

Let the modified
Jacobi theta function with argument $x$ and nome $p$ be defined by
\begin{equation*}
\ta(x)=\ta (x; p):= (x; p)_\infty (p/x; p)_\infty\,,\quad\quad
\ta (x_1, \ldots, x_m) = \prod^m_{k=1} \ta (x_k),
\end{equation*}
where $ x, x_1, \ldots, x_m \ne 0,\ |p| < 1,$ and $(x; p)_\infty=
\prod^\infty_{k=0}(1-x p^k)$.

These functions satisfy the following simple properties,
\begin{subequations}
\begin{equation}
\ta(x)=-x\,\ta(1/x),
\end{equation}
\begin{equation}\label{p1id}
\ta(px)=-\frac 1x\,\ta(x),
\end{equation}
\end{subequations}
and the three-term addition formula
(cf.\ \cite[p.~451, Example 5]{WW})
\begin{equation}\label{addf}
\ta(xy,x/y,uv,u/v)-\ta(xv,x/v,uy,u/y)
=\frac uy\,\ta(yv,y/v,xu,x/u).
\end{equation}
The relation \eqref{addf} is not obvious but crucial for the theory
of elliptic hypergeometric series. (Inductive proofs of summation
formulae usually involve functional equations or recursions
which are established by means of the theta addition formula.)
A proof of \eqref{addf} (and more general results)
is given in
\cite[Eq.~(3.4) being a special case of Lemma~3.3]{RS}. 

Now define the {\em theta shifted factorial} (or
{\em $q,p$-shifted factorial}) by
\begin{equation*}
(a;q,p)_n = \begin{cases}
\prod^{n-1}_{k=0} \ta (aq^k),& n = 1, 2, \ldots\,,\cr
1,& n = 0,\cr
1/\prod^{-n-1}_{k=0} \ta (aq^{n+k}),& n = -1, -2, \ldots.
\end{cases}
\end{equation*}
For compact notation, we write
\begin{equation*}
(a_1, a_2, \ldots, a_m;q, p)_n = \prod^m_{k=1} (a_k;q,p)_n.
\end{equation*}
Notice that $\ta (x;0) = 1-x$ and, hence, $(a;q, 0)_n = (a;q)_n$
is a {\em $q$-shifted factorial} in base $q$.

Observe that
\begin{equation}\label{pid}
(pa;q,p)_n=(-1)^na^{-n}q^{-\binom n2}\,(a;q,p)_n,
\end{equation}
which follows from \eqref{p1id}. 
A list of other useful identities for manipulating the
$q,p$-shifted factorials is given in \cite[Sec.~11.2]{GR}.

By definition, a function $g(u)$ is {\em elliptic}, if it is
a doubly-periodic meromorphic function of the complex variable $u$.

Without loss of generality, by the theory of theta functions,
one may assume that
\begin{equation*}
g(u)=\frac{\ta(a_1q^u,a_2q^u,\dots,a_sq^u;p)}
{\ta(b_1q^u,b_2q^u,\dots,b_sq^u;p)}\,z
\end{equation*}
(i.e., an abelian function of some degree $s$), 
for a constant $z$ and some
$a_1,a_2,\dots,a_s$, $b_1,\dots,b_s$, and $p,q$ with $|p|<1$,
where the {\em elliptic balancing condition} (cf.\ \cite{Sp1}), namely
\begin{equation*}
a_1a_2\cdots a_s=b_1b_2\cdots b_s,
\end{equation*}
holds. If one writes $q=e^{2\pi i\sigma}$, $p=e^{2\pi i\tau}$,
with complex $\sigma$, $\tau$, then $g(u)$ is indeed periodic in $u$
with periods $\sigma^{-1}$ and $\tau\sigma^{-1}$.
Keeping this notation for $p$ and $q$,
denote the {\em field of elliptic functions} over $\C$
of the complex variable $u$, meromorphic in $u$
with the two periods $\sigma^{-1}$ and $\tau\sigma^{-1}$ by
$\E_{q^u;q,p}$.

More generally, denote the {\em field of totally elliptic multivariate
functions} over $\C$ of the complex variables $u_1,\dots,u_n$,
meromorphic in each variable with equal periods,
$\sigma^{-1}$ and $\tau\sigma^{-1}$, of double periodicity, by
$\E_{q^{u_1},\dots,q^{u_n};q,p}$.

After these prerequisites, we are ready to turn to our elliptic
generalization of the $q$-binomial coefficient. (The corresponding
elliptic weight function will come out automatically.)
For indeterminants $a$, $b$, complex numbers $q$, $p$ (with $|p|<1$),
and nonnegative integers $n$, $k$,
define the {\em elliptic binomial coefficient} as follows
(this is exactly the expression for $w(\mathcal P((0,0)\to(k,n-k)))$
in \cite[Th.~2.1]{S}):
\begin{equation}\label{ellbc}
\begin{bmatrix}n\\k\end{bmatrix}_{a,b;q,p}:=
\frac{(q^{1+k},aq^{1+k},bq^{1+k},aq^{1-k}/b;q,p)_{n-k}}
{(q,aq,bq^{1+2k},aq/b;q,p)_{n-k}}.
\end{equation}
Note that this definition of the elliptic binomial coefficient
(which reduces to the usual $q$-binomial coefficient
after taking the limits $p\to 0$, $a\to 0$, and $b\to 0$,
in this order) is different from the much simpler one given in
\cite[Eq.~(11.2.61)]{GR}, the latter which is a straightforward
theta shifted factorial extension of the $q$-binomial coefficient
but actually {\em not} elliptic.
In fact, as pointed out in \cite{S}, it is not difficult to see
that the expression in
\eqref{ellbc} is {\em totally elliptic}, i.e.\
elliptic in each of $\log_qa$, $\log_qb$, $k$, and $n$
(viewed as complex parameters), with equal periods of double periodicity, 
which fully justifies the notion ``elliptic''. In particular,
$\left[\begin{smallmatrix}n\\k\end{smallmatrix}\right]_{a,b;q,p}
\in\E_{a,b,q^n,q^k;q,p}$.

It is immediate from the definition of \eqref{ellbc} that (for
integers $n,k$) there holds
\begin{subequations}\label{qbinrel}
\begin{equation}
\begin{bmatrix}n\\0\end{bmatrix}_{a,b;q,p}=
\begin{bmatrix}n\\n\end{bmatrix}_{a,b;q,p}=1,
\end{equation}
and
\begin{equation}
\begin{bmatrix}n\\k\end{bmatrix}_{a,b;q,p}=0,\qquad\text{whenever}\quad
k=-1,-2,\dots,\quad\text{or}\quad k> n.
\end{equation}
Furthermore, using the theta additional formula in \eqref{addf}
one can verify the following recursion formula for the
elliptic binomial coefficients:
\begin{equation}\label{rec}
\begin{bmatrix}n+1\\k\end{bmatrix}_{a,b;q,p}=
\begin{bmatrix}n\\k\end{bmatrix}_{a,b;q,p}+
\begin{bmatrix}n\\k-1\end{bmatrix}_{a,b;q,p}\,W_{a,b;q,p}(k,n+1-k),
\end{equation}
\end{subequations}
for nonnegative integers $n$ and $k$, where
\begin{equation}\label{Wdef}
W_{a,b;q,p}(s,t):=\frac{\ta(aq^{s+2t},bq^{2s},bq^{2s-1},aq^{1-s}/b,aq^{-s}/b)}
{\ta(aq^s,bq^{2s+t},bq^{2s+t-1},aq^{1+t-s}/b,aq^{t-s}/b)}q^t.
\end{equation}

Clearly, $W_{a,b;q,p}(s,0)=1$, for all $s$.
If one lets $p\to 0$, $a\to 0$, then $b\to 0$ (in this order), the
relations in \eqref{qbinrel} reduce to
\begin{equation*}
\begin{bmatrix}n\\0\end{bmatrix}_{q}=
\begin{bmatrix}n\\n\end{bmatrix}_{q}=1,
\end{equation*}
\begin{equation*}
\begin{bmatrix}n+1\\k\end{bmatrix}_{q}=
\begin{bmatrix}n\\k\end{bmatrix}_{q}+
\begin{bmatrix}n\\k-1\end{bmatrix}_{q}\,q^{n+1-k},
\end{equation*}
for positive integers $n$ and $k$ with $n\ge k$, which is
a well-known recursion for the $q$-binomial coefficients.

According to \eqref{w2W} we have for the small weights
\begin{equation}\label{wdef}
w_{a,b;q,p}(s,t):=\frac{W_{a,b;q,p}(s,t)}{W_{a,b;q,p}(s,t-1)}=
\frac{\ta(aq^{s+2t},bq^{2s+t-2},aq^{t-s-1}/b)}
{\ta(aq^{s+2t-2},bq^{2s+t},aq^{t-s+1}/b)}q,
\end{equation}
for $s,t\in\N$.

We refer to $w_{a,b;q,p}(s,t)$ (and to $W_{a,b;q,p}(s,t)$) as an
{\em elliptic weight function}.
Recall that in \cite{S} lattice paths in the integer
lattice $\Z^2$ were enumerated with respect to precisely this
weight function. A similar weight function was subsequently
used by A.~Borodin, V.~Gorin and E.M.~Rains in \cite[Sec.~10]{BGR}
(see in particular the expression obtained for
$\frac{w(i,j+1)}{w(i,j)}$ on p.~780 of that paper)
in the context of weighted lozenge tilings.

\subsection{An elliptic binomial theorem}\label{secell}

For the elliptic case, the commutation relations from
Definition~\ref{defwa} are particularly elegant and can be
formulated as follows. Recall (see the prerequisites we
just covered inbetween Equations \eqref{pid} and \eqref{ellbc})
that $\E_{a,b;q,p}$ denotes the field of totally
elliptic functions over $\C$, in the complex variables 
$\log_qa$ and $\log_qb$, with equal periods $\sigma^{-1}$,
$\tau\sigma^{-1}$ (where $q=e^{2\pi i\sigma}$, $p=e^{2\pi i\tau}$,
$\sigma,\tau\in\C$), of double periodicity.

\begin{definition}\label{defea}
For four noncommuting variables $x,y,a,b$, where $a$ and $b$ commute
with each other, and two complex numbers $q,p$ with $|p|<1$,
let $\C_{a,b;q,p}[x,y]$ denote
the associative unital algebra over $\C$, generated by $x$
and $y$, satisfying the following three relations:
\begin{subequations}\label{defeaeq}
\begin{align}
yx&=\frac{\ta(aq^3,bq,a/bq;p)}{\ta(aq,bq^3,aq/b;p)}qxy,\label{elleq}\\
xf(a,b)&=f(aq,bq^2)x,\label{xf}\\\label{yf}
yf(a,b)&=f(aq^2,bq)y,
\end{align}
\end{subequations}
for all $f\in\E_{a,b;q,p}$.
\end{definition} 
We refer to the variables $x,y,a,b$
forming $\C_{a,b;q,p}[x,y]$
as {\em elliptic-commuting} variables.

Notice that, in comparison with \eqref{defwaeq} the pair of positive
integers $(s,t)$ does not appear explicitly in the commutation relations
\eqref{defeaeq}.
It is easy to verify that the actions of $x$, respectively $y$,
on a weight $w(s,t)$ exactly correspond to shifts of the parameters
$a$ and $b$ as described in \eqref{xf} and \eqref{yf}.

The algebra $\C_{a,b;q,p}[x,y]$ reduces to
$\C_{q}[x,y]$ if one formally lets $p\to 0$, $a\to 0$,
then $b\to 0$ (in this order), while (having eliminated the nome $p$) 
relaxing the condition of ellipticity.

As in \eqref{defwaeq}, the relations in \eqref{defeaeq} are
well-defined as any expression in  $\C_{a,b;q,p}[x,y]$ can be
put in a unique canonical form regardless in which order the
commutation relations are applied for this purpose.

The generic weight-dependent noncommutative binomial theorem
in Theorem~\ref{wdbinth} reduces now to the following:

\begin{theorem}[Elliptic binomial theorem]\label{ebthm}
Let $n\in\N_0$. Then the following identity is valid in $\C_{a,b;q,p}[x,y]$:
\begin{equation}\label{eqbinth}
(x+y)^n=\sum_{k=0}^n\begin{bmatrix}n\\k\end{bmatrix}_{a,b;q,p}x^ky^{n-k}.
\end{equation}
\end{theorem}

\begin{remark}
As Erik Koelink has kindly pointed out, a result very similar to
Theorem~\ref{ebthm} has been proved in \cite[Eq.~(3.5)]{KNR},
as an identity in the elliptic $U(2)$ quantum group (or, equivalently,
the $\mathfrak h$-Hopf algebroid $\mathcal F_R(U(2))$). Nevertheless,
although both results involve a ``binomial'' expansion of noncommuting
variables, the correspondence between the two results is not entirely
clear. It is possible however, that such a correspondence would
be easier to make out for another (yet to be established) version
of elliptic binomial theorem in the framework of the more general
situation in Appendix~\ref{appa} with two weight functions $v$ and $w$
(where $v$ and $W$ contribute about the same number of factors).
\end{remark}

\subsection{Frenkel and Turaev's
${}_{10}V_9$ summation}\label{secV109}

{\em Elliptic hypergeometric series} are series $\sum_{k\ge 0}c_k$
where $c_0=1$ and $g(k)=c_{k+1}/c_k$ is an
elliptic function of $k$ with $k$ considered as a complex variable.

Elliptic hypergeometric series first appeared as
elliptic solutions of the Yang--Baxter equation
(or elliptic $6j$-symbols) in work by 
 E.~Date, M.~Jimbo, A.~Kuniba, T.~Miwa, and M.~Okado~\cite{DJKMO}
in 1987, and a decade later by
E.~Frenkel and V.~Turaev~\cite{FT}. The latter authors were the first
to find summation and transformation formulae satisfied by
elliptic hypergeometric series. In particular, by exploiting the
symmetries of the elliptic $6j$-symbols they derived the
(now-called) ${}_{12}V_{11}$ transformation.
By specializing this result they obtained the (now-called)
${}_{10}V_9$ summation (see also \cite[Eq.~(11.4.1)]{GR}),
an identity which is fundamental to the
theory of elliptic hypergeometric series.

\begin{proposition}[Frenkel and Turaev's ${}_{10}V_9$ summation]\label{propft}
Let $n\in\N_0$ and $a,b,c,d,e,q,p\in\C$ with $|p|<1$.
Then there holds the following identity:
\begin{align}\notag
\sum_{k=0}^n\frac{\ta(aq^{2k};p)}{\ta(a;p)}
\frac{(a,b,c,d,e,q^{-n};q,p)_k}
{(q,aq/b,aq/c,aq/d,aq/e,aq^{n+1};q,p)_k}q^k&\\\label{propfteq}
=
\frac{(aq,aq/bc,aq/bd,aq/cd;q,p)_n}
{(aq/b,aq/c,aq/d,aq/bcd;q,p)_n}&,
\end{align}
where $a^2q^{n+1}=bcde$.
\end{proposition}

For $p=0$ the ${}_{10}V_9$ summation reduces to
{\em Jackson's ${}_8\phi_7$ summation}~\cite[Eq.~(II.22)]{GR}.
Interestingly, the ${}_{10}V_9$ stands at the {\em bottom}
of the hierarchy of identities for elliptic hypergeometric series
(as one cannot send parameters of an elliptic hypergeometric series
to zero or infinity due to poles of infinite order).
The systematic study of elliptic hypergeometric series
commenced at about the turn of the millenium, after further
pioneering work of V.P.~Spiridonov and A.S.~Zhedanov~\cite{SZ},
and of S.O.~Warnaar~\cite{W}.

By the elliptic specialization of the
convolution formula in Corollary~\ref{wdconv1} one recovers
Frenkel and Turaev's~\cite{FT} ${}_{10}V_9$ summation
in the following form (where the requirement of $n$ and $m$
being nonnegative integers can be removed by repeated analytic
continuation):

\begin{corollary}\label{corftbinconv}
Let $n,m,k\in\N_0$ and $a,b,q,p\in\C$ with $|p|<1$.
Then there holds the following convolution
formula:
\begin{equation}\label{corftbinconveq}
\begin{bmatrix}n+m\\k\end{bmatrix}_{a,b;q,p}=
\sum_{j=0}^k\begin{bmatrix}n\\j\end{bmatrix}_{a,b;q,p}
\begin{bmatrix}m\\k-j\end{bmatrix}_{aq^{2n-j},bq^{n+j};q,p}
\prod_{i=1}^{k-j}W_{a,b;q,p}(i+j,n-j),
\end{equation}
where the elliptic binomial coefficients and the
weight function $W_{a,b;q,p}$ are defined in \eqref{qbinrel} and \eqref{Wdef}.
\end{corollary}

To see the correspondence with Proposition~\ref{propft},
replace the summation index $k$ in Equation~\eqref{propfteq} by $j$
and substitute the $6$-tuple of parameters $(a,b,c,d,e,n)$, appearing in
Equation~\eqref{propfteq} by
$(bq^{-n}/a,q^{-n}/a,bq^{1+n+m},bq^{-n-m+k}/a,q^{-n},k)$.
(This substitution is reversibel if $q^{-n}$ and $q^{-m}$
are treated as complex variables. This is fine,
as the terminating parameter has changed from $n$ to $k$.)
The resulting summation can be written, after
some elementary manipulations of theta shifted factorials,
exactly in the form of Equation \eqref{corftbinconveq}.

Interestingly, Corollaries~\ref{wdconv2} and \ref{wdconv3}
also yield essentially the same result, namely variants
of Frenkel and Turaev's ${}_{10}V_9$ summation. In particular,
from Corollary~\ref{wdconv2} one obtains the following identity:
\begin{equation*}
\begin{bmatrix}n+m\\n\end{bmatrix}_{a,b;q,p}=
\sum_{k=0}^m\begin{bmatrix}k+l-1\\l-1\end{bmatrix}_{a,b;q,p}
\begin{bmatrix}n+m-l-k\\n-l\end{bmatrix}_{aq^{l+2k},bq^{2l+k};q,p}
\prod_{i=0}^{n-l}w_{a,b;q,p}(i+l,k)
\end{equation*}
(which is the $(a,b,c,d,e,n)\mapsto(aq^l,bq^l,aq^{1+n+m},aq^{-n}/b,q^l,m)$
case of Equation~\eqref{propfteq}),
where the requirement of $n$ and $l$ being nonnegative integers
can be removed by repeated analytic continuation.

Whereas, from Corollary~\ref{wdconv3} one obtains
\begin{equation*}
\begin{bmatrix}n+m\\n\end{bmatrix}_{a,b;q,p}=
\sum_{l=0}^n\begin{bmatrix}l+k-1\\l\end{bmatrix}_{a,b;q,p}
\begin{bmatrix}n+m-l-k\\n-l\end{bmatrix}_{aq^{l+2k},bq^{2l+k};q,p}
\prod_{i=1}^{n-l}w_{a,b;q,p}(i+l,k)
\end {equation*}
(which is the ($k\mapsto l$, then)
$(a,b,c,d,e,n)\mapsto(bq^k,aq^k,bq^{1+n+m},bq^{-m}/a,q^k,n)$
case of Equation~\eqref{propfteq}),
where again the requirement of $m$ and $k$ being nonnegative integers
can be removed by repeated analytic continuation.

\begin{appendix}

\section{A generalization involving an additional
weight function}\label{appa}

A substantial amount of the analysis of Section~\ref{secgenw} can be
readily generalized to the situation where one not only has weights
$w(s,t)$ attributed to the horizontal steps but also additional
indeterminate weights $v(s,t)$ on the vertical steps.
More precisely, the weight of a vertical step in the
(first quadrant of the) integer lattice $\Z^2$ from $(s,t-1)$ to
$(s,t)$ shall be $v(s,t)$. 

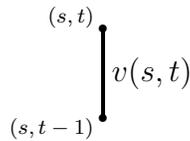
\centerline{
\unitlength1.2cm
\begin{picture}(0,1.8)
\linethickness{1pt}
\put(0,0.4){\circle*{.1}}
\put(-0.1,0.4){\makebox(0,0)[tr]{{\tiny $(s,t-1)$}}}
\put(0,0.4){\line(0,1){1}}
\put(0,1.4){\circle*{.1}}
\put(-0.1,1.4){\makebox(0,0)[br]{{\tiny $(s,t)$}}}
\put(0.1,.9){\makebox(0,0)[l]{$v(s,t)$}}
\end{picture}}

For instance, the path $P_0$ from Subsection~\ref{secwdbc}
now has the weight
\begin{align*}
w(P_0)=1\cdot 1\cdot v(2,1)\cdot W(3,1)\cdot
W(4,1)\cdot v(4,2)\cdot W(5,2)&\\
=v(2,1)w(3,1)w(4,1)v(4,2)w(5,1)w(5,2)&.
\end{align*}

Keeping the other notions from Section~\ref{secgenw}, let us describe
how the results look like in this generalized setting. 

First we have the following extension of the noncommutative
algebra $\C_w[x;y]$:

\begin{definition}\label{defvwa}
{}For  two doubly-indexed sequences of indeterminates
$(v(s,t))_{s,t\in\N}$ and\linebreak
$(w(s,t))_{s,t\in\N}$ let $\C_{v,w}[x,y]$
be the associative unital algebra over $\C$ generated
by $x$ and $v(0,1)\,y$, satisfying the following five relations:
\begin{subequations}\label{defvwaeq}
\begin{align}
yx&=w(1,1)\,xy,\\
x\,v(s,t)&=v(s+1,t)\,x,\\
x\,w(s,t)&=w(s+1,t)\,x,\\
y\,v(s,t)&=v(s,t+1)\,y,\\
y\,w(s,t)&=w(s,t+1)\,y,
\end{align}
\end{subequations}
for all $(s,t)\in\N^2$.
\end{definition}

As in Subsection~\ref{secwdbc} we define the big weight
$W(s,t)$ to be the product $\prod_{j=1}^tw(s,j)$ of the small $w$-weights.

Let the {\em double weight-dependent binomial coefficients}
be defined by
\begin{subequations}\label{vwbineq}
\begin{align}
&{}_{\stackrel{\phantom w}{\stackrel{\phantom w}v,w}}\!\!
\begin{bmatrix}0\\0\end{bmatrix}=1,\qquad
{}_{\stackrel{\phantom w}{\stackrel{\phantom w}v,w}}\!\!
\begin{bmatrix}n\\k\end{bmatrix}=0
\qquad\text{for $n\in\N_0$, and
$k\in-\N$ or $k>n$},\\
\intertext{and}\label{recvw}
{}_{\stackrel{\phantom w}{\stackrel{\phantom w}v,w}}\!\!
\begin{bmatrix}n+1\\k\end{bmatrix}=
&{}_{\stackrel{\phantom w}{\stackrel{\phantom w}v,w}}\!\!
\begin{bmatrix}n\\k\end{bmatrix}
v(k,n+1-k)
+{}_{\stackrel{\phantom w}{\stackrel{\phantom w}v,w}}\!\!
\begin{bmatrix}n\\k-1\end{bmatrix}
W(k,n+1-k)
\quad \text{for $n,k\in\N_0$}.
\end{align}
\end{subequations}

It is obvious that the double weight-dependent binomial coefficients
${}_{\stackrel{\phantom w}v,w}\!
\left[\begin{smallmatrix}n\\k\end{smallmatrix}\right]$ have again a
nice combinatorial interpretation in terms of weighted lattice paths.
The generating function $w_{v,w}$ with respect to the weights $v$ and $w$
of all paths from $(0,0)$ to $(k,n-k)$ is clearly
\begin{equation}\label{lpvwbc}
w_{v,w}(\mathcal P((0,0)\to (k,n-k)))=
{}_{\stackrel{\phantom w}{\stackrel{\phantom w}v,w}}\!\!
\begin{bmatrix}n\\k\end{bmatrix}.
\end{equation}

The noncommutative binomial theorem in Theorem~\ref{wdbinth}
extends to the following:

\begin{theorem}[Double weight-dependent binomial theorem]\label{vwdbinth}
Let $n\in\N_0$. Then the following identity is valid in $\C_{v,w}[x,y]$:
\begin{equation}\label{vwdbinthid}
(x+v(0,1)\,y)^n=\sum_{k=0}^n\,
{}_{\stackrel{\phantom w}{\stackrel{\phantom w}v,w}}\!\!
\left[\begin{matrix}n\\k\end{matrix}\right]x^ky^{n-k}.
\end{equation}
\end{theorem}
The proof of this theorem is a simple extension of the proof of
Theorem~\ref{wdbinth}.

The double weight-dependent binomial coefficients
${}_{\stackrel{\phantom w}v,w}\!
\left[\begin{smallmatrix}n\\k\end{smallmatrix}\right]$
have the advantage that they also cover 
various (generalizations of) important sequences.
In particular, the ($q$-){\em Stirling numbers of first kind} arise when
$v(s,t)=1-s-t$ (resp., $v(s,t)=(q^{s+t-1}-1)/(1-q)$) and
$w(s,t)=1$, for all $s,t\in\Z$, whereas
the ($q$-){\em Stirling number of second kind} arise
when $v(s,t)=s$ (resp., $v(s,t)=(1-q^s)/(1-q)$)
and $w(s,t)=1$, for all $s,t\in\Z$.

Notice that the algebra $\C_{v,w}[x,y]$ is not very interesting
when $w(s,t)=1$ for all $s,t\in\Z$.
It is certainly worthwhile to look for nontrivial (and ``nice'')
applications of Theorem~\ref{vwdbinth} for suitable choices
of the weight functions $v$ and $w$ (where none of them
is the identity function). This is not pursued further here,
the focus being laid on symmetric functions, elliptic hypergeometric
series and some basic hypergeometric specializations.

\section{Basic hypergeometric specializations}\label{appb}

Here some particularly attractive specializations of 
the elliptic weights $w_{a,b;q,p}(s,t)$ from
Section~\ref{secellbincoeffs} are considered.
The corresponding binomial coefficients and associated
commutation relations are given explicitly, while the summations
that are obtained by convolution are identified.

For some standard terminology related to basic hypergeometric series,
in particular the terms {\em balanced}, {\em well-poised}, 
{\em very-well-poised}, and the definition of an
{\em ${}_r\phi_s$ basic hypergeometric series}, the reader
is kindly referred to the classic text book \cite{GR}.

\subsection{The  balanced very-well-poised case}
If one specializes the elliptic weight function in \eqref{wdef}
by letting $p\to 0$ one obtains the following weights:
\begin{subequations}
\begin{equation}\label{wdefbalwp}
w_{a,b;q}(s,t)=\frac{(1-aq^{s+2t})(1-bq^{2s+t-2})
(1-aq^{t-s-1}/b)}
{(1-aq^{s+2t-2})(1-bq^{2s+t})(1-aq^{t-s-1}/b)}q,
\end{equation}
the associated big weights being
\begin{equation}\label{Wdefbalwp}
W_{a,b;q}(s,t)=\frac{(1-aq^{s+2t})(1-bq^{2s})(1-bq^{2s-1})
(1-aq^{1-s}/b)(1-aq^{-s}/b)}
{(1-aq^s)(1-bq^{2s+t})(1-bq^{2s+t-1})(1-aq^{1+t-s}/b)(1-aq^{t-s}/b)}q^t.
\end{equation}
\end{subequations}
The corresponding binomial coefficients are
\begin{equation}\label{balwpbc}
\begin{bmatrix}n\\k\end{bmatrix}_{a,b;q}=
\frac{(q^{1+k},aq^{1+k},bq^{1+k},aq^{1-k}/b;q)_{n-k}}
{(q,aq,bq^{1+2k},aq/b;q)_{n-k}},
\end{equation}
where we are using the suggestive compact notation
$(a_1, \ldots, a_m;q)_j = \prod_{l=1}^m(a_l;q)_j$ for products
of $q$-shifted factorials.

Now, in the unital algebra $\C_{a,b;q}[x,y]$ over $\C$ defined by
the following five commutation relations
\begin{subequations}\label{defbalwpaeq}
\begin{align}
yx&=\frac{(1-aq^3)(1-bq)(1-a/bq)}{(1-aq)(1-bq^3)(1-aq/b)}qxy,\\
xa&=qax,\\
xb&=q^2bx,\\
ya&=q^2ay,\\
yb&=qby,
\end{align}
\end{subequations}
the binomial theorem
\begin{equation}
(x+y)^n=\sum_{k=0}^n\begin{bmatrix}n\\k\end{bmatrix}_{a,b;q}x^ky^{n-k}
\end{equation}
holds.
Convolution yields {\em Jackson's balanced very-well-poised terminating
${}_8\phi_7$ summation}~\cite[Appendix~(II.22)]{GR}
(which of course is the $p\to 0$ case
of Frenkel and Turaev's ${}_{10}V_9$ summation,
see Corollary~\ref{corftbinconv} and the two identities
appearing thereafter).

\subsection{The balanced case}\label{appbbal}
If in \eqref{wdefbalwp} one lets $a\to 0$ the following weights
are obtained:
\begin{subequations}
\begin{equation}\label{wdefbal}
w_{0,b;q}(s,t)=\frac{(1-bq^{2s+t-2})}
{(1-bq^{2s+t})}q,
\end{equation}
the associated big weights being
\begin{equation}\label{Wdefbal}
W_{0,b;q}(s,t)=\frac{(1-bq^{2s})(1-bq^{2s-1})}
{(1-bq^{2s+t})(1-bq^{2s+t-1})}q^t.
\end{equation}
\end{subequations}
The corresponding binomial coefficients are
\begin{equation}\label{balbc}
\begin{bmatrix}n\\k\end{bmatrix}_{0,b;q}=
\frac{(q^{1+k},bq^{1+k};q)_{n-k}}
{(q,bq^{1+2k};q)_{n-k}}.
\end{equation}

Now, in the unital algebra $\C_{0,b;q}[x,y]$ over $\C$ defined by
the following three commutation relations
\begin{subequations}\label{defbalaeq}
\begin{align}
yx&=\frac{(1-bq)}{(1-bq^3)}qxy,\\
xb&=q^2bx,\\
yb&=qby,
\end{align}
\end{subequations}
the binomial theorem
\begin{equation}\label{bthbal}
(x+y)^n=\sum_{k=0}^n\begin{bmatrix}n\\k\end{bmatrix}_{0,b;q}x^ky^{n-k}
\end{equation}
holds.
Convolution yields the {\em balanced $q$-Saalsch\"utz summation}
~\cite[Appendix~(II.12)]{GR} which
is a summation for a balanced terminating ${}_3\phi_2$ series.
It is clear that reducing the weight in \eqref{wdefbal} yet further
by letting $b\to 0$ one arrives at the standard $q$-weight
connected to $q$-commuting variables. In this case convolution
gives the $q$-Chu--Vandermonde summation~\cite[Appendix~(II.6)/(II.7)]{GR},
a summation for a terminating ${}_2\phi_1$ series.

\subsection{The very-well-poised case}\label{appbvwp}
If in \eqref{wdefbalwp} one lets $b\to 0$ the following weights
are obtained:
\begin{subequations}
\begin{equation}\label{wdefwp}
w_{a,0;q}(s,t)=\frac{(1-aq^{s+2t})}
{(1-aq^{s+2t-2})}q^{-1},
\end{equation}
the associated big weights being
\begin{equation}\label{Wdefwp}
W_{a,0;q}(s,t)=\frac{(1-aq^{s+2t})}
{(1-aq^s)}q^{-t}.
\end{equation}
\end{subequations}
The corresponding binomial coefficients are
\begin{equation}\label{wpbc}
\begin{bmatrix}n\\k\end{bmatrix}_{a,0;q}=
\frac{(q^{1+k},aq^{1+k};q)_{n-k}}
{(q,aq;q)_{n-k}}q^{k(k-n)}.
\end{equation}

Now, in the unital algebra $\C_{a,0;q}[x,y]$ over $\C$ defined by
the following three commutation relations
\begin{subequations}\label{defwpaeq}
\begin{align}
yx&=\frac{(1-aq^3)}{(1-aq)}q^{-1}xy,\\
xa&=qax,\\
ya&=q^2ay,
\end{align}
\end{subequations}
the binomial theorem
\begin{equation}\label{bthvwp}
(x+y)^n=\sum_{k=0}^n\begin{bmatrix}n\\k\end{bmatrix}_{a,0;q}x^ky^{n-k}
\end{equation}
holds.
Convolution yields the {\em very-well-poised terminating
${}_6\phi_5$ summation}~\cite[Appendix~(II.21)]{GR}.
It is clear that the further $a\to\infty$ limit
of \eqref{wdefwp} leads again to the classical case of
$q$-commuting variables (whereas $a\to0$ leads to the same
with $q$ replaced by $q^{-1}$). 

\end{appendix}

\end{document}